\documentclass
{amsart}
\usepackage{amsmath,amsthm,amsfonts,amssymb,epsfig}
\usepackage{setspace}
\usepackage{amsmath,amssymb,amsthm,euro} 
\usepackage{bbm}
\usepackage{stmaryrd}
 \usepackage{tensor}

\newcommand\F{\mathcal F}
\newcommand\I{\mathbbm{1}}
\renewcommand{\P}{\mathbb{P}}
\newcommand{\Q}{\mathbb{Q}}

\renewcommand{\rho}{\varrho}

\newcommand\eps{\varepsilon}

\newcommand\N{\mathbb{N}}
\newcommand\E{\mathbb{E}}

\newcommand\RR{\mathbb{R}}

\newtheorem{theorem}{Theorem}[section]  
\newtheorem*{theorem*}{Theorem BD}

\newtheorem{corollary}[theorem]{Corollary}

\newtheorem{lemma}[theorem]{Lemma}

\theoremstyle{definition}

\theoremstyle{example}

\DeclareMathOperator{\conv}{conv}
\DeclareMathOperator{\Var}{MV}
\DeclareMathOperator{\MV}{MV}

\DeclareMathOperator{\sign}{sign}

\title[Riemann-integration and the Bichteler--Dellacherie  Theorem]{Riemann-integration and a new proof of the Bichteler--Dellacherie  Theorem}

\author{M. Beiglb\"ock$^{\ddagger}$}
\author{P. Siorpaes$^{\dagger\ddagger}$}

\thanks{${}^{\dagger}$  University of Vienna, Faculty of Mathematics,  \emph{email:} mathias.beiglboeck@univie.ac.at, \textit{Corresponding author}\\ 
${}^{\ddagger}$ University of Vienna, Faculty of Mathematics,  \emph{email:} pietro.siorpaes@univie.ac.at}

\thanks{The first  author thanks the Austrian science fund FWF for support through project p21209. Both authors thank Josef Teichmann and Pianta Giacomo for useful comments.}

\begin{document}
\maketitle
\begin{abstract}
We give a new  proof of the celebrated Bichteler--Dellacherie Theorem, which states that a process $S$ is a good integrator if and only if it is the sum of a local martingale and a finite-variation process. As a corollary, we obtain a characterization of semimartingales along the lines of classical Riemann integrability.

\bigskip

\noindent\emph{Keywords:} Bichteler--Dellacherie Theorem, semimartingale decomposition, good integrators.\\
\emph{Mathematics Subject Classification (2010):} 60G05 
\end{abstract}

\section{Introduction}
The Bichteler--Dellacherie  theorem basically asserts that one can integrate with respect to a process $S$ iff $S$ is a semimartingale, i.e., the sum of a local martingale and a finite-variation process; in this paper we provide a new proof of this celebrated result, together with a new characterization of semimartingales.

The Doob-Meyer decomposition theorem leads to the following reformulation of the Bichteler--Dellacherie theorem: a bounded process allows for a good integration theory iff it is (locally) the difference of two submartingales.
This is analogous to  the deterministic case, where one can integrate
with respect to a function $f$ iff $f$ can be written as a difference
of two increasing functions.  We find that this analogy is sound, as the simple proof in the deterministic set-up can be reinterpreted to establish the Bichteler--Dellacherie  theorem in full generality.

As a corollary, we obtain that  semimartingales can be characterized by Riemann-sums in the following way:
a c\`adl\`ag adapted process $(S_t)_{t\in [0,1]}$ is  a semimartingale iff 
 for every bounded adapted continuous process $H$ the sequence of Riemann-sums
\begin{align}
\label{riemsum}
\textstyle \sum_{i=0}^{2^n-1} H_{\frac i{2^n}} (S_{\frac{i+1}{2^n}}-S_{\frac i{2^n}}) 
 \end{align}
converges in probability. This observation emphasizes the viewpoint that semimartingales are the stochastic equivalent of processes of finite variation.

We notice that this is in remarkable contrast to a fact which Meyer \cite{Me81}   attributes to Jeulin: there are continuous processes which are not semimartingales for which \eqref{riemsum} holds for all  integrands $H$ of the type $H_t=f(t,S_t)$, where $f$ is a bounded continuous function.

\section{Definitions, assumptions and main statement}
Throughout this article we consider a finite time horizon $T$, which wlog we take to be equal to $1$, and a filtered probability space $(\Omega,\mathcal{F},\mathbb{F},P)$. We assume that the filtration $\mathbb{F}=(\mathcal{F}_t)_{t\in [0,1]}$ satisfies  the usual conditions of right continuity and saturatedness.
A \emph{simple integrand} is a stochastic process $H = (H_t)_{t\in (0,1]} $ of the form
\begin{equation}
\label{H}
\textstyle H=\sum_{i=1}^k H^i \I_{(\tau_i,\tau_{i+1}]} , 
\end{equation}
where $k$ is a finite number, $0\leq \tau_1\leq \ldots\leq \tau_{k+1}\leq 1$ are stopping times, and $H^i$ are bounded $\F_{\tau_i}$-measurable random variables.
The vector space of simple integrands will be denoted by $\mathcal{S}$, and will be endowed with the sup norm 
  \begin{align}
\label{supnorm}
\textstyle \|H\|_\infty:= \|\sup_{t\in [0,1]} |H_t| \,
 \|_{L^\infty(\P)}.
\end{align}

Given an adapted (real-valued) process $S=(S_t)_{t\in[0,1]}$ and a simple integrand $H$ as in (\ref{H}), it is natural to define the (It\^o) integral $\mathcal{I}_S(H)$ of $H\in \mathcal{S}$ with respect to $S$ as the random variable
 \begin{equation}
\label{riemannintegral}
\textstyle\mathcal{I}_S(H):= \sum_{i=1}^k H^i ( S_{\tau_{i+1}}-S_{\tau_{i}}).
\end{equation}
 This defines the integral as a linear operator $\mathcal{I}_S$ from the normed space $\mathcal{S}$ to the topological vector space $L^0(\P)$ (the space of all random variables,
with the metrizable topology of convergence in probability).  A process $S$ is  then called a \emph{good integrator} if $\mathcal{I}_S: \mathcal{S} \to  L^0(\P)$ is continuous, i.e.\ if $H^n\in \mathcal S, \|H^n\|_\infty \to 0$ implies that $\mathcal{I}_S(H^n)$ goes to $0$ in probability as $n\to \infty$. 

 It is easy to show that (locally) square integrable martingales and processes of finite variation are good integrators. It is also true that any (local) martingale is a good integrator, although this requires a little more work; we refer to \cite{Ed90}  for an elementary proof of this result which does not make use of the structure of local martingales in continuous time. 
 
The converse result  is of key importance to stochastic analysis, as it characterizes the processes $S$ for which one can build a powerful integration theory. 
This is the object of the following well known theorem, commonly known as the Bichteler-Dellacherie Theorem. 

\begin{theorem*}
\label{BichtelerDellacherie}
Let $(S_t)_{0\leq t \leq 1}$ be a c\`adl\`ag adapted process. 
If $\mathcal{I}_S: \mathcal{S} \to  L^0(\P)$ is continuous then $S$ can be written as a sum of a  c\`adl\`ag local martingale and a c\`adl\`ag adapted process of finite variation.
\end{theorem*}

 Theorem BD has a long history, tracing back to the Rennes school of Metivier and Pellaumail (see for example \cite{Pe73,MePe77}), and then evolving in the Strasbourg school of  Meyer; it was first published in its present form in \cite{Me79c} and, independently, \cite{Bi79,Bi81}. Mokobodzki deserves particular credit (see for instance the discussion in \cite{DeMe82B}); however since the result is usually baptized after Bichteler and Dellacherie, we stick to this name.

We emphasize that  the definition of good integrators requires that the integrands are adapted. Simply dropping this assumption would amount to considering all simple processes
that are adapted to the constant filtration  $\mathcal{G}_t:= \mathcal{F}_1$, $0\leq t \leq 1$. Since $(\mathcal{G}_t)$-local martingales are constant, Theorem BD implies that every $(\mathcal{G}_t)$-good integrator has paths of finite variation. So, if one chooses to consider integrands which are not necessarily adapted (predictable) one is left with an unreasonably small class of integrators.

Since submartingales provide a filtration-dependent stochastic
 equivalent of increasing functions, we believe that the
 following reformulation of Theorem BD is quite intuitive.
\begin{theorem}\label{MainNewResult}
Let $S=(S_t)_{0\leq t\leq 1}$ be a bounded c\`adl\`ag adapted process. If $S$ is a good integrator then it is locally the difference of two c\`adl\`ag submartingales
\end{theorem}
We recall that a process defined on $[0,1]$ satisfies a property \emph{locally} if, \ for each $\eps>0$, there exist a $[0,1]\cup \{\infty\}$-valued
 stopping time $\rho$ such that  $S^{\rho}$ satisfies that property
 and  $\P(\rho=\infty)\geq 1-\eps$ (by definition $\mathcal{F}_{\infty}:=\mathcal{F}_1$).
 
\medskip
Our main contribution consists in a new proof of Theorem
\ref{MainNewResult}. Its equivalence with Theorem BD easily follows
from the Doob--Meyer decomposition theorem, of which in recent years
simple and elementary proofs have been obtained: we refer the reader
to  \cite{Ba96,Ja05, BeScVe11}.

\medskip
The most popular accounts on Theorem BD employ functional analytic machinery and change of measure techniques, as in Dellacherie-Meyer \cite{DeMe82B}. A modern version of this argument is given in Protter \cite{Pr05}; specifically, one applies a variant of the Hahn-Banach separation theorem (due to Yan \cite{Ya80}) to construct an equivalent measure $\Q$ under which the good integrator $S$ is a quasimartingale. Using the Theorems of Rao and Doob-Meyer it follows that $S$ is a $\Q$-semimartingale; finally, Girsanov's Theorem implies that $S$ is also a $\P$-semimartingale.

Some accounts use the notion of a good integrator (or a similar concept) as  starting point to develop the theory of stochastic integration  (see e.g. \cite{Met77, MePe80, Met82, Pr05}). This greatly simplifies many proofs; however, the useful link to the classical approach based on semimartingales is only obtained a posteriori by Theorem BD.
Since the proof of  Theorem BD given in the present article does not rely on stochastic calculus nor Girsanov's theorem the equivalence of the two approaches could be established from the start, enabling further simplifications.

For an alternative approach to Theorem BD, based on an orthogonal decomposition, see Lowther \cite{Lo11}.
A different idea is developed in \cite{BeScVe11} which, like the present paper, has an elementary proof based on discrete time arguments, and does not use change of measure techniques; it thus seems interesting to describe this approach more closely. 

In \cite{BeScVe11}, Theorem BD is obtained as a corollary  of the fact that  every bounded process $S$ satisfying a certain weak No-Arbitrage condition is a semimartingale (\cite[Theorem 1.6]{BeScVe11}). To prove the latter result, the authors take the discrete time Doob-decomposition $S=M^n+A^n$ of $S$ restricted to the dyadic times of $n^{\textrm{th}}$-generation and, repeatedly applying the  No-Arbitrage property, show that the sequences $(M^n)_n, (A^n)_n$ can be controlled on suitably chosen (random) intervals $[0, \tau_n]$. 
 Using carefully chosen  convex combinations 
 it is then possible to pass to limits and obtain processes $M,A$  on $[0,\tau]$ such that $M$ is a  martingale, $A$ has finite variation (but is not necessarily predictable), and $\tau$ is an arbitrarily ``big'' stopping time. 
  This comes with necessity to develop quite intricate estimates on the approximations as well as a somewhat complex limiting procedure which takes into account the approximating processes $(M^n)_n, (A^n)_n$ and the intervals $[0, \tau_n]$ \emph{simultaneously}.

\medskip

 This paper is organized as follows. 
After recalling  Rao's Theorem  in the next section, we provide the proof of Theorem \ref{MainNewResult} in Section 4. 
 The fact that Theorem \ref{MainNewResult} implies  Theorem BD is shown in detail in Section 5. Finally, in Section 6 we discuss certain
 ramifications of the Theorem BD (including the characterization \eqref{riemsum}).

We conclude this section with some definitions that will be used throughout the paper.
 As it is customary, we will denote by $X^+$ ($X^-$) the positive (negative) part of a random variable $X$, and by $D_n$ the $n$-th dyadic partition of $[0,1]$, i.e. $D_n:=\{0, 1/2^n,2/2^n,\ldots,1 \}$.
We will not be picky about the difference between functions and their equivalence classes.
Given a simple integrand $H$,  $H \cdot S$  denotes the process given by
$ (H \cdot S)_t:=\mathcal{I}_{S^t}(H).$ 
Recall that a family $F \subseteq L^0(\P)$ is bounded if for every $\eps>0$
there exists a constant $C$ such that $\P(|X|\geq C)\leq \eps$ for
every $X\in F$. A simple proof, analogous to the one for normed spaces, shows that a linear operator from a normed
space to $L^0(\P)$ is continuous iff  it is bounded, i.e., it maps bounded sets into bounded sets; we will use this fact without further mention.

\section{Quasimartingales}

To prove that a given function $f=f(t)$ can be written as a difference
of two increasing functions, one would typically show that $f$ has
finite variation. This has an analogue in the stochastic world; to
state it, we recall the notion of quasimartingale.

Let $S=(S_t)_{0\leq t\leq 1}$ be an adapted process such that $S_t\in L^1$ for all $t\in [0,1]$.
Given a partition $\pi=\{0=t_0< t_1<\ldots< t_n=1\}$ of $[0,1]$, the \emph{mean variation} of $S$ \emph{along} $\pi$ is defined as 
$$ \textstyle\Var(S, \pi)=\E\sum_{t_i\in \pi} \big|\E[S_{t_i}-S_{t_{i+1}}|\F_{t_i}]\big|.$$
 Note that the mean variation along $\pi$ is an increasing function of
 $\pi$, i.e.\ we have $\Var (S,\pi)\leq \Var (S,\pi')$, whenever $\pi'
 $ is a partition refining $\pi$: this follows from the conditional Jensen inequality  $ | \E(X|\mathcal{G}) | \leq \E( |X|\, |\mathcal{G})$.

By definition, $S$ is a \emph{quasimartingale}\footnote{ The study of
  quasimartingales goes back to Fisk \cite{Fi65},  Orey \cite{Or67}, Rao \cite{Ra69}, and Stricker \cite{St77}.} if it is adapted, $S_t\in L_1, t\in [0,1]$ and
the \emph{mean variation} $$\textstyle \Var(S):= \sup_{\pi} \Var(S,
\pi)$$ of $S$ is finite.
We will use that if $S$ is bounded and c\`adl\`ag then trivially $\Var(S)=\lim_{n} \Var (S,D_{n})$.

 The stochastic analogue of the fact that a function has bounded variation if and only if it can be written as a difference of two increasing functions is then provided by the following characterization of quasimartingales, usually known as  Rao's theorem; its standard proof, found in most stochastic calculus textbooks (e.g. 
\cite[Chapter 3, Theorem 17]{Pr05}, \cite[Chapter 6, Theorem 41.3]{RoWi00}, and \cite[Chapter 6, Theorem 40]{DeMe82B}), is both short and elementary.
\begin{theorem}\label{Rao's theorem}
A c\`adl\`ag process $S$ is a quasimartingale if and only it has a decomposition $S = Y- Z$ as the difference of two c\`adl\`ag submartingales $Y$ and $Z$.
\end{theorem}

In dealing with the mean variation of stopped processes the following lemma is useful.
\begin{lemma}\label{SuperLemma} Let $S$ be a bounded process.
Given a partition $\pi$ and a  stopping time  $\rho$  define $\rho+:=\inf\{t \in \pi: t\geq \rho\}$. Then
\begin{align}\label{NiceRep}\Var(S^{\rho+}, \pi) =\textstyle{\E \sum_{t_i \in \pi} \I_{\{ t_i< \rho\}} \big|\E[S_{t_{i+1}}-S_{t_{i}}|\F_{t_i}]\big|}\end{align} 
and $|\Var(S^{\rho+}, \pi)-\Var(S^{\rho}, \pi)| \leq 2\|S\|_\infty.$
\end{lemma}
\begin{proof}
To obtain \eqref{NiceRep}, observe that for each $t_i\in \pi$ 
$$\E[S^{\rho+}_{t_{i+1}}-S^{\rho+}_{t_{i}}|\F_{t_i}]=
\E[(S_{t_{i+1}}-S_{t_{i}})\I_{\{ t_i<\rho\}}|\F_{t_i}]=\I_{\{ t_i<\rho\}}\E[(S_{t_{i+1}}-S_{t_{i}})|\F_{t_i}].$$
Given processes $S', S''$ the conditional Jensen inequality implies
$$\textstyle |\Var(S',\pi)-\Var(S'', \pi)|\leq \E \sum_{t_i\in \pi}|(S'_{t_{i+1}}-S'_{t_{i}})-(S''_{t_{i+1}}-S''_{t_{i}})|.$$ Applying this to $S'=S^{\rho}, S''=S^{\rho+}$ concludes the proof, as the only (possibly) non-zero term in the above sum is the one for which $\rho\in [t_i, t_{i+1}).$
 \end{proof}

\section{The technical core}

The aim of this section is to establish Theorem \ref{MainNewResult}.
To motivate our approach, assume that a continuous  function $f:[0,1]\to \RR$ gives rise to a  Riemann-Stieltjes integral
$$\textstyle h\mapsto \int h(t)\,  df(t)$$
which is continuous on the space of piecewise constant functions $h:[0,1] \to \RR$, endowed with the sup norm. Then $f$ has finite total variation; indeed the sequence of  piecewise constant functions $$\textstyle h^n:=\sum_{t_i\in D_n} \I_{(t_i, t_{i+1}]}\sign\big(f(t_{i+1})-f(t_{i})\big)$$  is bounded uniformly
and
$$\textstyle \int_{0}^1 h^n\, df =
\sum_{t_i\in D_n} |f(t_{i+1})-f(t_{i})| $$
converges to the total variation of $f$.
The subsequent proof is merely a translation of this standard argument
to the stochastic setting, where the integrands are assumed to be adapted.

\begin{lemma}\label{BoundedByTrading} 
Let $S=(S_t)_{0\leq t\leq 1}$ be a c\`adl\`ag bounded adapted good integrator. 
 Then for every $\eps>0$ there exist a constant $C$ and a sequence of
 $[0,1]\cup \{\infty\}$-valued  stopping times  $(\rho_n)_n$ such that $\P(\rho_n=\infty)\geq 1-\eps$ and $\Var (S^{\rho_n},D_n)\leq C$.
\end{lemma}
\begin{proof}

Since $S$ is a good integrator, given $\eps>0$ there exists $C>0$ so that for all simple processes 
$H$ with $\| H\|_\infty\leq 1$ we have $\P((H\cdot S)_1\geq C- 2\|S\|_\infty )\leq \eps.$ 
 For each $n$ we define the simple process $H^{n}$ and the stopping time $\rho_{n}$ as
\begin{align*}
\textstyle H^{n}&:= \textstyle\sum_{t_i\in D_n} \I_{(t_i, t_{i+1} ]} \sign \big( \E[S_{t_{i+1}}-S_{t_i}|\F_{t_i}] \big) ,\\
\rho_{n}&:=\inf \{ t\in {D_n}: (H^{n}\cdot S)_t\geq C-2\|S\|_\infty\}.
\end{align*}
Notice that, on the set $\{\rho_n<\infty\}$,
  $$(H^n 1_{(0,\rho_n]})\cdot S=(H^n\cdot S)^{\rho_n} \text{ satisfies }  (H^n\cdot S)^{\rho_n}_1\geq C-2\|S\|_\infty ,$$ and thus $\P(\rho_n=\infty)\geq 1-\eps$.
Moreover, since the increments of $S$ are bounded by $2\|S\|_\infty\, $,
$C\geq (H^n\cdot S)^{\rho_n}_1$ holds, so we find, with the help of
lemma \ref{SuperLemma} , 
\begin{align*} C\geq  \E (H^n\cdot S)^{\rho_n}_1 =\ &\E\sum_{t_i\in D_n}\I_{\{t_i< \rho_n\}}  \sign \big( \E[S_{t_{i+1}}-S_{t_i}|\F_{t_i}] \big) (S_{t_{i+1}}-S_{t_i})=\\
=\ &\E\sum_{t_i\in D_n} \I_{\{t_i< \rho_n\}}\Big|\E[ (S_{t_{i+1}}-S_{t_i})|\F_{t_i}]\Big|=\Var(S^{\rho_n},D_n).\qedhere
\end{align*}
\end{proof}

Given that $\Var (S^{\rho_n},D_k)\leq C$ for every $k\leq n$, it is
desirable to define an ``accumulation stopping time'' $\rho$ of the stopping times $(\rho_n)_ n$, so that $\Var (S^{\rho},D_k)\leq C$ will hold for every $k$, proving that $S^{\rho}$ is a quasimartingale.  
 Ideally, we would want $\rho$ to be ``as big as the $\rho_n$'', and yet such that $\rho\leq \rho_{n_k}$ holds for some subsequence $n_k$. This is not quite possible; however, after rephrasing  the previous inequality as
$\I_{[0,\rho]}\leq \I_{[0,\rho_{n_k}]},$ we can soften this requirement (passing to forward convex combinations instead of subsequences), thus making it compatible with $\rho$ being  ``big''.
A similar technique is also used in \cite[Proposition 3.6]{BeScVe10}.

\begin{lemma}\label{LimitStopping}
Assume that $(\rho_n)_n$ is a sequence of $[0,1]\cup\{\infty\}$-valued stopping times such that $\P(\rho_n=\infty)\geq 1-\eps,$  $n\geq 1$ for some $\eps>0$. Then there exists a stopping time $\rho$ and for each $n\geq 1$ convex weights $\mu_n^n,\ldots, \mu_{N_n}^n$ such that\footnote{We note that the constant $2$ in \eqref{SmallerSense} can be replaced by $1+\delta$, for $\delta>0$ in which case one is only guaranteed to find $\rho$ satisfying $\P(\rho= \infty)\geq 1-\eta \eps$ for $\eta>(1-(1+\delta)^{-1})^{-1} $. But we do not need this.} 
 $\P(\rho =\infty)\geq 1-3\eps$
 and for all $n\geq 1$
\begin{align}\label{SmallerSense}
\textstyle \I_{[0,\rho]}\leq 2 \sum_{k=n}^{N_n} \mu_k^n \I_{[0,\rho_k]}.
\end{align}
\end{lemma}

\begin{proof}[Proof of Lemma \ref{LimitStopping}] Recall the following classical result by Mazur:
if $(f_n)_n$ is a bounded sequence in a Hilbert space then there exist vectors $ g_n\in \conv(f_n, f_{n+1},\ldots)$, $n\geq 1$ such that $(g_n)_n$ converges in Norm.\footnote{This can be seen as a consequence of weak compactness combined with the fact that weak and strong closure coincide for convex sets. Alternatively one may simply pick the elements $g_n$ to have (asymptotically) minimal norm in $\conv(f_n, f_{n+1},\ldots)$, $n\geq 1$.}
\medskip
We apply this to the random variables $X_n=\I_{\{\rho_n=\infty\}}\in L^2(\P), n\geq 1$ to obtain for each $n$ convex weigths $\mu_n^n,\ldots, \mu_{N_n}^n$ such that 
$$\textstyle Y_n:= \mu_n^n X_n+\ldots+\mu_n^{N_n} X_{N_n}$$
converges to some random variable $X$ in $L^2(\P)$.
Relabeling sequences if necessary, we assume that the convergence holds also almost surely.

From $X\leq 1$ and $\E[X]\geq 1-\eps$ we deduce that $\P(X< 2/3)< 3 \eps$. Since $\P(\lim_m Y_{m}\geq 2/3)> 1-3\eps$, by Egoroff's theorem we deduce that there exists a set $A$ with $\P(A)\geq 1-3\eps$ such that $Y_n \geq 1/2$ on the set $A$,  for all $n$ greater or equal than some $n_0\in\N$, which we can assume to be equal to $1$. 

We now define the desired stopping time $\rho$ by
$$\textstyle{\rho=\inf_{n\geq 1} \inf\{t: \mu_n^n  \I_{[0,\rho_n]}(t)+\ldots+\mu_n^{N_n}  \I_{[0,\rho_{N_n}]}(t)< 1/2\}}.$$
Then clearly (\ref{SmallerSense}) holds, and  from $A\subseteq \{\rho=\infty\}$ we obtain $\P(\rho=\infty)\geq 1-3\eps.$
\end{proof}

We are now in the position to complete the proof of Theorem \ref{MainNewResult}

\begin{proof}[Proof of Theorem \ref{MainNewResult}.]
Given $\eps>0$, pick $C$, $(\rho_n)_n$ and $\rho$ according to Lemma \ref{BoundedByTrading} resp.\ Lemma \ref{LimitStopping}.
Fixing $n\geq 1$ we obtain from \eqref{SmallerSense} that
\begin{align}\label{ShorterTrick}\E\!\! \sum_{t_i\in D_n}\!\!\I_{\{t_i<\rho\}} \Big|\E[S_{t_{i+1}}-S_{t_i}\|\F_{t_i}]\Big| \leq  
 2 \E\!\!\!\sum_{t_i\in D_n}\!\sum_{k=n}^{N_n}\mu_k^n \I_{\{t_i < \rho_k\} } \Big| 
 \E[S_{t_{i+1}}-S_{t_i}|\F_{t_i}]\Big|.
 \end{align}
 By Lemma \ref{SuperLemma}, $\Var(S^\rho, D_n)$ differs from the left side of \eqref{ShorterTrick} by at most $2\|S\|_\infty$. Applying Lemma \ref{SuperLemma} once more, the right side of \eqref{ShorterTrick} is bounded by 
 $$\textstyle 2\sum_{k=n}^{N_n}\mu_k^n (\Var(S^{\rho_k},D_n) +2\|S\|_\infty) \leq 2C+4\|S\|_\infty.$$
 Combining these facts and letting $n\to \infty$ we conclude  $\MV(S^\rho)\leq 2 C + 6\|S\|_\infty$. By Rao's theorem \ref{Rao's theorem} this yields Theorem \ref{MainNewResult}.
\end{proof}
\section{Every good integrator is a Semimartingale} 
In this section, for the convenience of the reader, we show in detail how Theorem BD follows from  Theorem \ref{MainNewResult}; all arguments are however quite standard. For a proof of the following lemma one can also consult \cite[Proposition 4.25(b), Chapter 1]{JacodShir:02}.

\begin{lemma}\label{LocalProperty}
 Let a process $S$ be locally a 
 semimartingale. Then $S$ is a 
 semimartingale.
\end{lemma}
\begin{proof}

If $S=(S_t)_{t\in [0,1]}$ is locally a semimartingale there exists a sequence $(\sigma_n)_n$ of stopping times such that $\P(\sigma_n\leq 1)\to 0$ and, for each $n$, a local martingale $M_n$ and a process $A_n$ of finite variation such that $S^{\sigma_n}=M_n+A_n$. 
By passing to a subsequence $n_i$ s.t. $\P(\sigma_{n_i}\leq 1)\leq 2^{-i}$ and then replacing $\sigma_k$ with $\rho_k:= \inf_{i\geq k} \sigma_{n_i}$ we can assume that $(\sigma_n)_n$  is increasing (indeed $(\rho_k)_k$ is increasing and $\P(\rho_{k}\leq 1)\leq 2^{-(k-1)}\to 0$). 
Since $S^{\sigma_n}=M_{n}^{\sigma_n}+A_{n}^{\sigma_n}$ equals  $(S^{\sigma_{n+1}})^{\sigma_n}=M_{n+1}^{\sigma_n}+A_{n+1}^{\sigma_n}$, 
\begin{align*}
S&= S^{\sigma_1}+(S^{\sigma_2}-S^{\sigma_1})+(S^{\sigma_3}-S^{\sigma_2})+ \ldots\\
&= [ M^{\sigma_1}_1+( M_2^{\sigma_2}- M^{\sigma_1}_2)+ \ldots] +[ A_1^{\sigma_1}+( A_2^{\sigma_2}- A_2^{\sigma_1})+ \ldots]=:M+A,
\end{align*}
where for each $(t,\omega)$ only one term in each sum is non-zero.
Since $A$ is of finite variation and $M$ is locally a local martingale, and thus a local martingale (see \cite[Lemma 1.35(a), Chapter 1]{JacodShir:02}), $S$ is a semimartingale.  
\end{proof}
Notice that the previous lemma implies that Theorem \ref{BichtelerDellacherie} applies also when the time index $[0,1]$ is replaced with $[0,\infty)$. 
Recall that a process $X$ is of class $D$ if the family $\{X_{\sigma}:\sigma\mbox{ stopping time}\}$ is uniformly integrable.
\begin{lemma}
\label{LocClassD}
Let $S=(S_t)_{0\leq t \leq 1}$ be a c\`adl\`ag submartingale. Then $S$ is locally of class D.
\end{lemma}
\begin{proof}
Define the stopping time $T_n:= \inf\{ t\in [0,1]: |S_t|\geq n \}.  $ Then if $\sigma$ is an arbitrary stopping time we have that 
$|S^{T_n}_{\sigma}|\leq n+ |S_{1\wedge T_n}| $. By the optional sampling theorem\footnote{For a proof see \cite[Theorem 3.22]{KaSh91} or \cite[Theorem II.77.1]{RoWi00}.}  $S_{1\wedge T_n}$ is integrable, showing that $\{S_\sigma^{T_n}:\sigma \mbox{ stopping time}\}$ is uniformly integrable.
\end{proof}

\begin{proof}[Proof of Theorem BD]
We note that $S$ can be written as the sum two adapted processes, one of finite
variation and one locally bounded: indeed, since $S$ is
c\`adl\`ag,  $\Delta S_t:=S_t-S_{t-}$ and $J_t:=\sum_{0< s \leq t}
\Delta S_t \I_{\{|\Delta S_t |\geq 1\}}$ are well defined (the sum
defining $J_t(\omega)$ is finite for each $t,\omega$). Since $J$ has
finite variation and is adapted, and $S-J$ has bounded jumps, $S=J+(S-J)$ is a
decomposition as required. Notice that $J$ is a c\`adl\`ag good
integrator (since it has finite variation), and so such is
$S-J$. Thus by localizing and using Lemma \ref{LocalProperty} we may assume without loss of generality that $S$ is bounded. By Theorem \ref{MainNewResult} it follows that
$S$ is locally the difference of  two c\`adl\`ag
submartingales. By Lemma \ref{LocClassD} and the Doob-Meyer
decomposition theorem  $S$ is locally a local semimartingale, and
thus applying twice Lemma \ref{LocalProperty}  we obtain that $S$ is a semimartingale.
\end{proof}

\section{Ramifications of the Bichteler-Dellacherie theorem}
In this section we prove that Riemann integrators are good
integrators, and somewhat strengthen Theorem BD.

It is well known that, in the definition of good integrators, the
space $\mathcal S$ can be replaced by the subset of
\emph{elementary integrands}, which consists of all processes $H$ of
the form
\begin{align}\label{ElRep}\textstyle
  H=\sum_{i=1}^k H^i \I_{(t_i,t_{i+1}]} , \end{align} where 
$ t_i$ are \emph{deterministic} times such that 
$0\leq t_1< \ldots< t_{k+1}=1$, 
and each $H_i$ is bounded
$\F_{t_i}$-measurable. In Lemma \ref{EleSuff} we prove this fact
in a slightly stronger form, which will be useful in proving Corollary \ref{NewChar}.

Let $\mathcal{E}_{D_n}$ be the space of 
all processes $H$ of
the form 
\begin{align}\label{ElRep2}
\textstyle
  H=\sum_{i=0}^{2^n -1} H^i \I_{ ( \frac{i}{2^n}, \frac{i+1}{2^n} ] } , \end{align}
where, for each $i=1,...,2^n-1$, $H^i$ is bounded and $\F_{\frac{i-1}{2^n}}$-measurable
(not only $\F_{\frac{i}{2^n}}$-measurable!), and  $H^0=0$; then, define  $\mathcal{E}_{D}:=
\bigcup_{n\geq 1} \mathcal{E}_{D_n} $. 

\begin{lemma}\label{EleSuff} 
  Let $S$ be an adapted process which is  right continuous in
  probability.
 Then $\mathcal{I}_S:
  \mathcal{S} \to L^0(\P)$ is a continuous operator if and only if its
  restriction to  $\mathcal{E}_{D}$ is continuous.
\end{lemma}
\begin{proof}
  We have to show that if $\mathcal{I}_S$ is a bounded operator on
  $\mathcal E_D$, then it is also a bounded operator on $\mathcal S$.
  Given $\eps >0$, pick $C>0$ such that $\P(|\mathcal{I}_S(K)| > C) < \eps$
  for every process $K\in \mathcal E_D$ satisfying $\|K\|_\infty\leq
  1$. Let $H$ be a simple integrand as in (\ref{H}) and satisfying $\|H\|_\infty\leq
  1$, and define
the stopping times   $$\sigma_i^n:=1\wedge (i+2)/2^n \text{ on } \{i/2^n
<\tau_i \leq (i+1)/2^n \}.$$
 Then the process
$\textstyle K^n:=\sum_{i=1}^k H^i \I_{(\sigma_i^n,\sigma_{i+1}^n]}  $
is actually in $ \mathcal{E}_{D_n} $: this follows from the fact that
the stopping times $(\sigma_i^n)_i$ have values in $D_n$ and satisfy
$\tau_i +1/2^n \leq \sigma_i^n$, while $H^i$ is $\F_{\tau_i}$-measurable. 
Moreover $\mathcal{I}_S (K^n)$ converges to $ \mathcal{I}_S (H)$ in
probability (since $S$ is right continuous) and so, taking $n$ big
enough,  it follows that
$$\P(|\mathcal{I}_S(H)| > C)\leq \P(|\mathcal{I}_S(H)-\mathcal{I}_S
(K^n)| > C)+\P(|\mathcal{I}_S (K^n)| > C)< 2\eps .$$
Since $C$ was chosen independent of $H \in \mathcal S$, this
proves that  $\mathcal{I}_S$ is bounded on   $\mathcal S$.
\end{proof}

The previous lemma could be  reformulated as follows:
a cadlag adapted process $S$ is a good integrator iff 
$\mathcal{I}_S(H^n) \to 0$ in probability 
whenever  $\|H^n\|_\infty \to 0$ and $H^n \in \mathcal{E}_{D_n}$ for all $n$.

As a corollary of Theorem BD, we obtain that  semimartingales can be characterized by Riemann-sums.
Indeed, if $S$ is a semimartingale, the stochastic dominated convergence theorem implies that, for every left-continuous (resp. cadlag) adapted process $H$, the random variables
$$ \sum_{\tau_i\in \pi_n} H_{\tau_i} (S_{\tau_{i+1}}-S_{\tau_{i}})$$
converge in probability (to $\mathcal{I}_S(H)$, resp.  $\mathcal{I}_S(H_{-})$) as $n \to \infty$, for
any sequence  $(\pi_n)_n$ of random partitions whose mesh is going
to $0$. 
Conversely, we find that  this property characterizes semimartingales. Indeed, define a c\`adl\`ag adapted process $(S_t)_{0\leq t\leq 1}$ to be a  \emph{Riemann integrator} if for every bounded adapted continuous process\footnote{As mentioned in the introduction,  there are continuous processes which are not semimartingales for which \eqref{riemsum2} holds for all  integrands $H$ of the type $H_t=f(t,S_t)$, where $f$ is a bounded continuous function. On the other hand, every continuous deterministic $S$ for which \eqref{riemsum2} holds for all  integrands $H$ of the type $H_t=f(S_t)$ (with $f$  bounded continuous) is a function of finite variation (see \cite[Prop. 4.1]{Str81}).}  $H$ the sequence of random variables
\begin{align}\label{riemsum2} \textstyle
 \sum_{i=0}^{2^n-1} H_{\frac i{2^n}} (S_{\frac{i+1}{2^n}}-S_{\frac i{2^n}}) 
 \end{align}
converges\footnote{In fact, to obtain Corollary \ref{NewChar} it would be sufficient to require that the sequence in \eqref{riemsum2} is bounded in $L^0$ (by the same proof).} in probability as $n \to \infty$. Then, the following holds:

\begin{corollary}\label{NewChar}
Every Riemann integrator is a semimartingale.
\end{corollary}

To prove Corollary \ref{NewChar}  we need some additional definitions. Consider the Banach
space $L^{\infty}(\Omega; C^0([0,1]))$ of all bounded continuous
processes $(H_t)_{t\in [0,1]}$, endowed with the sup norm
(\ref{supnorm}).
 Let $X$ be the subspace constituted by the processes which are
 adapted; this is
 a closed subspace, and hence a Banach space with the induced norm.
 Finally, define the linear continuous operator $\mathcal{I}_S^n:X \to L^0(\P)$ by
$$\textstyle \mathcal{I}_S^n(K):=\mathcal{I}_S(K^{D_n}), \text{ where }  K^{D_n}:=\sum_{i=0}^{2^n -1} K_{\frac{i} {2^n} }
 \I_{(\frac{i} {2^n} ,\frac{i+1} {2^n}    ]} .$$
  By definition, $S$ is a Riemann integrator if, for every $K\in X$, 
$\mathcal{I}_S^n(K)$ converges in probability as $n \to \infty$.
The Banach-Steinhaus theorem\footnote{Which is also commonly called ``the uniform
  boundedness principle''.}
 \cite[Theorem
2.6]{Ru91}
then yields the following:
\begin{lemma}
\label{BanachSteinhausLemma}
  If $S$ is a Riemann good integrator, then for every $\eps>0$ there is
  some $C>0$ such that $\P(\mathcal{I}_S^n(K)\geq C)\leq \eps$ for all
  $n\geq 1$ and all  continuous adapted processes $K$ such that  $\|K\|_\infty\leq 1$.
\end{lemma}

It is now fairly straightforward to show that every Riemann integrator is a good
integrator.

\begin{proof}[Proof of Corollary \ref{NewChar}.]
Let $H\in \mathcal E_{D_n}$ be as in (\ref{ElRep2}) and satisfy  $\|H\|_\infty\leq
  1$.
Define a process $K$ by declaring it equal to $H^i$  at time $
t=i/2^n$, for $0\leq i\leq 2^n-1$, and equal to zero at time $1$, and extending it to $t\in [0,1]$ by
affine interpolation.
Then   $K$ is a continuous \emph{adapted} process such that  $\|K\|_\infty\leq
  1$ and $K^{D_n}=H$. Since $n$ was arbitrary and  $\mathcal{I}_S(H)=\mathcal{I}_S(K^{D_n})=\mathcal{I}_S^n(K)$ , $\mathcal{I}_S$ is  bounded on  $\mathcal E_D= \bigcup_{n\geq 1} \mathcal{E}_{D_n} $ by Lemma \ref{BanachSteinhausLemma}. Then
 Lemma \ref{EleSuff} shows that $S$ is a
  good integrator, and so Theorem BD implies that $S$ is a semimartingale.
\end{proof}

\end{document}